\documentclass{amsart}
\newtheorem{theorem}{Theorem}[section]
\newtheorem{lemma}[theorem]{Lemma}
\newtheorem{proposition}[theorem]{Proposition}
\newtheorem{corollary}[theorem]{Corollary}
\theoremstyle{definition}
\newtheorem{definition}[theorem]{Definition}
\newtheorem{example}[theorem]{Example}

\newtheorem{remark}[theorem]{Remark}
\numberwithin{equation}{section}
\usepackage[all]{xy}
\DeclareMathOperator{\Po}{Po}
\DeclareMathOperator{\Cl}{Cl}
\DeclareMathOperator{\disc}{disc}
\DeclareMathOperator{\Gal}{Gal}
\DeclareMathOperator{\Int}{Int}
\DeclareMathOperator{\Ker}{Ker}
\DeclareMathOperator{\Coker}{Coker}

\DeclareMathOperator{\res}{res}
\begin{document}
\title{Relative P\'olya group and
P\'olya Dihedral extensions of $\mathbb{Q}$}
\author{Abbas Maarefparvar}
\address{Department of Mathematics, Tarbiat Modares University, 14115-134, Tehran, Iran}
\curraddr{}
\email{a.maarefparvar@modares.ac.ir}
\thanks{}
\author{Ali Rajaei$^{*}$}
\address{Department of Mathematics, Tarbiat Modares University, 14115-134, Tehran, Iran}
\curraddr{}
\email{alirajaei@modares.ac.ir}
\thanks{$^{*}$Corresponding author}


\dedicatory{}

\begin{abstract}
A number field with trivial P\'olya group \cite{Cahen-Chabert's book} is called a P\'olya field. We define ``relative P\'olya group $\Po(L/K)$"  for  $L/K$ a finite extension of number fields, generalizing the P\'olya group.
 Using cohomological tools in \cite{Brumer-Rosen}, we compute some relative P\'olya groups. As a consequence, we generalize Leriche's results in \cite{Leriche 2014} and prove the triviality of relative P\'olya group  for the Hilbert class field of $K$.
Then we generalize our previous results \cite{Maarefparvar-Rajaei} on P\'olya $S_3$-extensions of $\mathbb{Q}$ to dihedral extensions of $\mathbb{Q}$ of order $2l$, for $l$  an odd prime. 
We also improve Leriche's upper bound in \cite{Leriche 2013} on the number of ramified primes in P\'olya $D_l$-extensions of $\mathbb{Q}$ and prove that for a real (resp. imaginary) P\'olya $D_l$-extension of $\mathbb{Q}$ at most $4$ (resp. $2$) primes ramify.
\end{abstract}

\maketitle
 
\vspace{.2cm} {\noindent \bf{Keywords:}}~ integer valued polynomial, P\'olya field, P\'olya group, relative P\'olya group, dihedral extension.

\vspace{.2cm} {\noindent \bf{Notation.}}~ The following notation will be used throughout this article: 

For a number field $K$, $I(K)$, $P(K)$, $\Cl(K)$,  $\mathcal{O}_K$, $h(K)$, $U_K$, $w_K$, $H(K)$, $\Gamma(K)$ and $D_K$ denote the group of fractional ideals, group of principal fractional ideals, ideal class group, ring of integers, class number, group of units, Dirichlet rank of group of units, Hilbert class field, genus field and discriminant of $K$, respectively. 

For a finite extension $L/K$ of number fields, $\mathcal{N}_{L/K}:\Cl(L) \rightarrow \Cl(K)$ denotes the induced  morphism by the ideal norm morphism $N_{L/K}:I(L) \rightarrow I(K)$. Likewise  ${\epsilon}_{L/K}:\Cl(K) \rightarrow \Cl(L)$ denotes the transfer of ideal classes induced by the morphism $j_{L/K}: \mathfrak{a} \in I(K)  \mapsto \mathfrak{a} \mathcal{O}_L \in  I(L)$.

For a prime ideal $\mathfrak{p}$  of $K$ and a prime ideal  $\mathfrak{B}$ of $L$ above $\mathfrak{p}$, denote the ramification index and residue class degree of $\mathfrak{B}$ over $\mathfrak{p}$ by $e(\mathfrak{B}/\mathfrak{p})$ and $f(\mathfrak{B}/\mathfrak{p})$, respectively.

 
Finally, $l$ is an odd prime number, and  for integer $n \geq 3$, $C_n$ and $D_n$ denote the cyclic group of order $n$ and the dihedral group of order $2n$, respectively.

\section{Introduction} \label{section, Introduction}
Historically, the study of P\'olya fields dates back to P\'olya's results on integer valued polynomials \cite{Polya}.  For a number field $K$, with ring of integers $\mathcal{O}_K$, the ring of integer valued  polynomials on $\mathcal{O}_K$ is defined as follows:
\begin{align*}
\Int(\mathcal{O}_K)=\{f \in K[x] \mid f(\mathcal{O}_K) \subseteq \mathcal{O}_K\}.
\end{align*}

One can show that $\Int(\mathcal{O}_K)$ is free as an $\mathcal{O}_K$-module, but an explicit  $\mathcal{O}_K$-basis for $\Int(\mathcal{O}_K)$ may be difficult to write down, and P\'olya \cite{Polya} was interested in those fields $K$ for which  Int($\mathcal{O}_K$) has an $\mathcal{O}_K$ basis which exactly one member from each degree. Such basis, if it exists, is called a \textit{regular basis}. 

For a nonnegative integer $n$, denote the subset of $K$ formed by $0$ and the leading coefficients of integer valued polynomials of degree $n$ on $\mathcal{O}_K$ by $\mathfrak{J}_n(K)$.  This is a fractional ideal of $\mathcal{O}_K$: $\mathfrak{J}_n(K) \subseteq (n!) ^{-1}\mathcal{O}_K$,  see
\cite[Section 2]{zantema}. P\'olya proved that $\Int(\mathcal{O}_K)$ has a regular basis if and only if all the ideals $\mathfrak{J}_n(K)$ are principal, see \cite[Satz I]{Polya}.  Immediately after P\'olya, Ostrowski \cite{Ostrowski} proved that $\Int(\mathcal{O}_K)$ has a regular basis if and only if all the ideals 
\begin{align*}
\Pi_q(K)=:\prod_{\substack{\mathfrak{m}\in Max(\mathcal{O}_K)\\ N_{K/ \mathbb{Q}}(\mathfrak{m})=q}} \mathfrak{m}
\end{align*}
are principal.


Note that for a Galois number field $K$, existence of a regular basis for $\Int(\mathcal{O}_K)$ is equivalent to principality of $\prod_{i=1}^g \mathfrak{P}_i$, where $\mathfrak{P}_1,\mathfrak{P}_2,\dots,\mathfrak{P}_g$ are all distinct prime ideals of $K$ above a ramified prime $p$, see  \cite[Section 1]{zantema}.

\begin{definition} \cite{zantema}
A number field $K$ is called \textit{P\'olya}, if the $\mathcal{O}_K$-module $\Int(\mathcal{O}_K)$ admits a regular basis.
\end{definition}

\begin{definition} \cite[Definition II.3.8]{Cahen-Chabert's book} \label{definition, Polya group}
For a number field $K$, the \textit{P\'olya-Ostrowski group}  or \textit{P\'olya group} of $K$ is the subgroup $\Po(K)$ of $\Cl(K)$ generated by the classes of the ideals $\mathfrak{J}_n(K)$. It is easy to see that this is same as the subgroup generated by the classes of the ideals $\Pi_q(K)$ as well, see \cite[Proposition II.3.9]{Cahen-Chabert's book}.
\end{definition}

Now let $K$ be a Galois extension of $\mathbb{Q}$ with Galois group $G$. One can show that the \textit{Ostrowski  ideals} $\Pi_q(K)$ freely  generate the ambiguous ideals $I(K)^G$, see \cite[Section 2]{Brumer-Rosen}. Thus $\Po(K)$ is the subgroup of $\Cl(K)$ generated by the classes of ambiguous ideals. For quadratic fields, Hilbert proved:

\begin{proposition} \cite[Theorem 106]{Hilbert's book} \label{proposition, Zantema's result for quadratic Polya field}
Let $K$ be a quadratic field and denote the number of ramifid primes in $K/\mathbb{Q}$ by $s_K$. If $K$ is real and the fundamental unit of $K$ has positive norm, then $\# \Po(K)=2^{s_k-2}$. Otherwise $\# \Po(K)=2^{s_k-1}$.
\end{proposition}

The above theorem of Hilbert has been generalized by Zantema:

\begin{proposition}  \cite[ Section 3, page 9]{zantema} \label{proposition, Zantema's exact sequence}
Let $K/\mathbb{Q}$ be a Galois extension with Galois group $G$. Denote the ramification index of a prime $p$ in $K$ by $e_p$ .Then the following sequence is exact:
\begin{equation} \label{equation, Zantema's exact sequence}
\{0\} \longrightarrow H^1(G,U_K) \longrightarrow \bigoplus _{\text{p prime}} \mathbb{Z} / e_p \mathbb{Z} \longrightarrow \Po(K) \longrightarrow \{0\}.
\end{equation}
\end{proposition}

\begin{remark} \label{remark, if (h(K),[K:Q])=1, then K is Polya but not conversely}
By Zantema's exact sequence \eqref{equation, Zantema's exact sequence}, for a Galois number field $K$,  $\#\Po(K)$  divides $\prod _{p \text{ prime}} e_p$, so if $gcd([K:\mathbb{Q}],h(K))=1$ then $K$ is P\'olya, but not conversely. For instance, as we will see, in Example \eqref{example, first example of $D_7$-extension of Q} there is a P\'olya $D_7$-extension of $\mathbb{Q}$ with class number 7.

\end{remark}


The reader is refered to \cite{Cahen-Chabert's book,Chabert 2001,Chabert I,Chabert II, HR1,HR2,Leriche 2011,Leriche 2013,Leriche 2014,Maarefparvar's Thesis,Maarefparvar-Rajaei,zantema}, for some results on P\'olya fields and P\'olya groups.

\section{Relative P\'olya group} \label{section, Generalization of Zantema's exact sequence}
Using Brumer-Rosen's method in \cite{Brumer-Rosen}, we can generalize  Zantema's result in \cite[Section 3]{zantema} to finite Galois extensions of number fields and find a generalization of exact sequence  \eqref{equation, Zantema's exact sequence}.
First we recall the definition of the \textit{relative P\'olya group} \cite{Maarefparvar's Thesis} (also independently defined by Chabert \cite{Chabert I}):

\begin{definition} \cite[Definition 4.1]{Maarefparvar's Thesis}
Let $L/K$ be a finite extension of number fields. The \textit{relative P\'olya group} of $L$ over $K$, is the subgroup of $\Cl(L)$ generated by the classes of the ideals  $\Pi_{\mathfrak{P}^f}(L/K)$, where $\mathfrak{P}$ is a prime ideal of $K$, $f$ is a positive integer and $\Pi_{\mathfrak{P}^f}(L/K)$ is defined as follows:
\begin{align*}
\Pi_{\mathfrak{P}^f}(L/K)=\prod_{\substack{\mathfrak{M}\in Max(\mathcal{O}_L)\\ N_{L/ K}(\mathfrak{M})=\mathfrak{P}^f}} \mathfrak{M}.
\end{align*}
We denote the relative P\'olya group of $L$ over $K$ by $\Po(L/K)$. In particular, $\Po(L/\mathbb{Q})=\Po(L)$ and $\Po(L/L)=\Cl(L)$.
\end{definition}
Now let $L/K$ be a Galois extension with Galois group $G$. It is easily seen that the set of all  ideals $\Pi_{\mathfrak{P}^f}(L/K)$ is a set of free generators for the ambiguous ideals $I(L)^G$, see \cite[Proof of Proposition 2.2]{Brumer-Rosen}.
On the other hand $P(L)^G=I(L)^G \cap P(L)$. Hence in this case $\Po(L/K)=I(L)^G/P(L)^G$.
As in \cite[Section 3]{zantema}, we define:
\begin{align} \label{equation, the map psi similar to Zantema's definition}
\psi : I(L)^G \rightarrow \bigoplus _{\mathfrak{P} \text{ is a prime of K}} \mathbb{Z} / e_{\mathfrak{P}} \mathbb{Z} \nonumber  \\
(\psi(\prod_{i=1}^m \Pi_{\mathfrak{P}^{f_i}}(L/K)^{t_i}))_{\mathfrak{P_i}}:=t_i (\mathrm{mod}\, e_{\mathfrak{P_i}}),
\end{align}
where $m$ is a positive integer, $t_i$'s are integers and $e_\mathfrak{P}$ denotes the ramification index of $\mathfrak{P}$ in $L/K$. It is clear that $\psi$ is a group epimorphism and one can show that $\Ker(\psi)=I(K)$. Hence we get the following exact sequence:
\begin{equation} \label{equation, exact sequence is obtained by the defined map psi}
\{0 \} \longrightarrow I(K) \longrightarrow I(L)^G \longrightarrow \bigoplus _{\mathfrak{P} \text{ is prime of K}} \mathbb{Z} / e_{\mathfrak{P}} \mathbb{Z} \longrightarrow \{0 \}.
\end{equation}
Following Brumer-Rosen \cite{Brumer-Rosen}, Consider the following exact sequence:
\begin{equation*}
\{0 \} \longrightarrow U_L \longrightarrow L^* \longrightarrow P(L) \longrightarrow \{0 \}.
\end{equation*}
Taking Cohomology and using Hilbert's theorem 90 \cite{Hilbert's book}, we get the exact sequence 
\begin{equation*}
\{0 \} \longrightarrow U_K \longrightarrow K^* \longrightarrow P(L)^G \longrightarrow H^1(G,U_L) \longrightarrow   \{0 \}
\end{equation*}

Equivalently, the following sequence is exact:
\begin{equation} \label{equation, exact sequence similar to Brumer-rosen's result}
\{0 \} \longrightarrow P(K) \longrightarrow P(L)^G \longrightarrow H^1(G,U_L) \longrightarrow   \{0 \}.
\end{equation}

Now we can generalize Zantema's result in \cite[Section 3]{zantema}:
\begin{theorem} \label{theorem, generalization of the Zantema's exact sequence}
Let $L/K$ be a finite Galois extension of number fields with Galois group $G$. Then the following sequence is exact:
{\small
\begin{equation} \label{equation, main exact sequence}
\{0 \} \rightarrow \Ker({\epsilon}_{L/K}) \rightarrow H^1(G,U_L) \rightarrow \bigoplus _{\mathfrak{P} \text{ is prime of K}} \mathbb{Z} / e_{\mathfrak{P}} \mathbb{Z} \rightarrow \frac{\Po(L/K)}{{\epsilon}_{L/K}(\Cl(K))} \rightarrow \{0 \}.
\end{equation}}

\end{theorem}
\begin{proof}
We have the following commutative diagram with exact rows of abelian groups:
{\small
\begin{displaymath}
\xymatrix{
& \{0 \} \ar[d]  &  \{0 \} \ar[d] \\
\{0 \} \ar[r] & P(K) \ar[r] \ar[d] & I(K) \ar[r] \ar[d] & \Cl(K)  \ar[r] \ar[d]_{{\epsilon}_{L/K}} & \{0\} \\
\{0\} \ar[r] &P(L)^G \ar[r] \ar[d] & I(L)^G \ar[r] \ar[d]_{\psi} & I(L)^G/P(L)^G \ar[r] & \{0\} \\
& H^1(G,U_L) \ar[d] & \bigoplus _{\mathfrak{P} \mid \disc(L/K)} \mathbb{Z} / e_{\mathfrak{P}} \mathbb{Z} \ar[d] \\
& \{0\} & \{0 \}}
\end{displaymath}}

The first column is the exact sequence \eqref{equation, exact sequence similar to Brumer-rosen's result}, and the second column is the exact sequence \eqref{equation, exact sequence is obtained by the defined map psi}. Hence by the snake lemma, we find an exact sequence as follows:
\begin{equation*}
\{0 \} \rightarrow \Ker({\epsilon}_{L/K}) \rightarrow H^1(G,U_L) \rightarrow \bigoplus _{\mathfrak{P} \mid \disc(L/K)} \mathbb{Z} / e_{\mathfrak{P}} \mathbb{Z}   \rightarrow \Coker({\epsilon}_{L/K}) \rightarrow   \{0 \}.
\end{equation*}
Since $L/K$ is a Galois extension, $\Po(L/K)=I(L)^G/P(L)^G$ and the statement is proved.
\end{proof}

\begin{remark} \label{remark, if M/N is Galois and gcd(h(M),[M:N])=1, then Po(M/N)=j(Cl(N))}
By exact sequence \eqref{equation, main exact sequence},  order of the quotient group $\frac{\Po(L/K)}{{\epsilon}_{L/K}(\Cl(K))}$ divides a power of $[L:K]$. On the other hand, $\Po(L/K)$ is a subgroup of $\Cl(L)$. Hence if $\gcd(h(L),[L:K])=1$, then $\Po(L/K)={\epsilon}_{L/K}(\Cl(K))$. 
\end{remark}

\begin{corollary} \label{corollary, if h(N) and order of H^1 are relatively prime then Cl(N) is embedded  in Po(M/N)}
Let $L/K$ be a finite Galois extension of number fields with Galois group $G$. If $h(K)$ is relatively prime to $[L:K]$, then $\Cl(K)$ is embedded in $\Po(L/K)$. In particular, $h(K)$ divides $h(L)$. Moreover, we get a generalization of Zantema's exact sequence \eqref{equation, Zantema's exact sequence} as follows:
\begin{equation*}
\{0\} \rightarrow H^1(G,U_L)  \rightarrow \bigoplus_{\mathfrak{P} | \disc(L/K)} \mathbb{Z}/e_{\mathfrak{P}} \mathbb{Z} \rightarrow \Po(L/K)/\Cl(K) \rightarrow \{ 0\}.
\end{equation*}
\end{corollary}

\begin{proof}
Since $gcd(h(K),[L:K])=1$, $\mathcal{N}_{L/K} \, o \, {\epsilon}_{L/K}:\bar{\mathfrak{a}} \in \Cl(K) \mapsto \bar{\mathfrak{a}}^{[L:K]} \in \Cl(K)$ is injective. Hence ${\epsilon}_{L/K}$ is injective and the statement follows from Theorem \eqref{theorem, generalization of the Zantema's exact sequence}.
\end{proof}

\begin{remark} 
By the above proof, if $gcd(h(K),[L:K])=1$, then ${\epsilon}_{L/K}$ is injective, even for non-Galois extensions $L/K$. Indeed this is a particular case of Leriche's result in \cite[Proposition 6.3]{Leriche 2013}.
\end{remark}

\begin{remark}
If we are just interested in $\#\Po(L/K)$ of a Galois extensions $L/K$ (without the hypothesis that $h(K)$ and  $[L:K]$ are coprime), then interchanging the horizontal and vertical rows in the diagram in proof of Theorem \eqref{theorem, generalization of the Zantema's exact sequence}, gives 
\begin{equation*}
\#\Po(L/K)=\frac{h(K). \prod_{\mathfrak{P} | disc(L/K)} e_{\mathfrak{P}}}{\#H^1(G,U_L)}.
\end{equation*}
This has been done in \cite[Proposition 4.4]{Chabert I}.
\end{remark}

\begin{corollary} \label{corollary, if M/N is Galois and every ideal class of N extended to M is principal then we get a generalization of Zantema's exact sequence}
Let $L/K$ be a finite Galois extension of number fields with Galois group $G$. If every ideal class of $K$ extended to $L$ is principal, then 
\begin{equation*}
\{0 \} \rightarrow \Cl(K) \rightarrow H^1(G,U_L) \rightarrow \bigoplus _{\mathfrak{P} \text{ is prime of K}} \mathbb{Z} / e_{\mathfrak{P}} \mathbb{Z}
 \rightarrow \Po(L/K) \rightarrow   \{0 \}
\end{equation*}
is exact.
\end{corollary}

\begin{proof}
In this case, $\Ker({\epsilon}_{L/K})=\Cl(K)$ and the statement follows from Theorem \eqref{theorem, generalization of the Zantema's exact sequence}. Note that this generalizes Zantema's exact sequence \eqref{equation, Zantema's exact sequence}.
\end{proof}

\begin{corollary} \label{corollary, if M/N is Galois and unramified then H^1 is embedded in Cl(N) and Po(M/N) isomorph. to j(Cl(N))}
Let $L/K$ be a finite Galois extension of number fields with Galois group $G$. If all finite places of $K$ are unramified in $L$, then $H^1(G,U_L)$ is embedded in $\Cl(K)$  and $\Po(L/K)\simeq {\epsilon}_{L/K}(\Cl(K))$.
\end{corollary}

\begin{proof}
If all finite places of $K$ are unramified in $L$, by exact sequence \eqref{equation, main exact sequence}, we have $H^1(G,U_L) \simeq \Ker({\epsilon}_{L/K})$ is a subgroup of $Cl(K)$ and the quotient group $\frac{\Po(L/K)}{{\epsilon}_{L/K}(\Cl(K))}$ would be trivial .
\end{proof}


\begin{corollary} \label{corollary, if M/N is Galois and unramified and every ideal class of N extended to M is principal, then Cl(N) isomorphic to H^1 and Po(M/N) is trivia}
Let $L/K$ be a finite Galois extension of number fields with Galois group $G$. If every ideal class of $K$ extended to $L$ is principal and all finite places of $K$ are unramified in $L$, then $\Cl(K) \simeq H^1(G, U_L)$ and $\Po(L/K) = \{0\}$. In particular, $\Po(H(K)/K)$ is trivial, for the Hilbert class field $H(K)$ of $K$.
\end{corollary}

\begin{proof}
Immediately follows from Corollaries \eqref{corollary, if M/N is Galois and every ideal class of N extended to M is principal then we get a generalization of Zantema's exact sequence} and \eqref{corollary, if M/N is Galois and unramified then H^1 is embedded in Cl(N) and Po(M/N) isomorph. to j(Cl(N))}.
\end{proof}

Leriche \cite{Leriche 2014} proved that the Hilbert class field $H(K)$ of $K$ is a P\'olya field.
Using the following lemma (part $(ii)$, for $P=\mathbb{Q}$), we conclude that Corollary \eqref{corollary, if M/N is Galois and unramified and every ideal class of N extended to M is principal, then Cl(N) isomorphic to H^1 and Po(M/N) is trivia} is a generalization of this result \cite[Corollary 3.2]{Leriche 2014}:

\begin{lemma} \label{lemma, for P sub N sub M, if M/N is Galois are Galois, then Po(M/P) is contained in Po(M/N)}
Let $P \subseteq K \subseteq L$ be a tower of finite extensions of number fields. 

\begin{itemize}
\item[$(i)$]
If $K/P$ and $L/P$ are Galois extensions, then $\epsilon_{L/K}(\Po(K/P)) \subseteq \Po(L/P)$.
\item[$(ii)$]
If $L/K$ is a Galois extension, then $\Po(L/P) \subseteq \Po(L/K)$.
\end{itemize}

\end{lemma}

\begin{proof}
$(i)$ Let $\mathfrak{p}$, $\beta$ and $\gamma$ be respectively primes of $P$, $K$ and $L$ such that $\gamma | \beta | \mathfrak{p}$. Since $K/P$ and $L/P$ are Galois extensions, one has
\begin{equation*}
j_{L/K}(\Pi_{\mathfrak{p}^f}(K/P))=(\Pi_{\mathfrak{p}^{f^{\prime}}}(L/P))^{e(\gamma/\beta)} \in \Po(L/P),
\end{equation*}
 where $f=f(\beta/\mathfrak{p})$ and $f^{\prime}=f(\gamma/\mathfrak{p})$.

$(ii)$ Let
\begin{align*}
\mathfrak{p} \mathcal{O}_K&= \beta_1^{e_1} \beta_2^{e_2} \dots \beta_g^{e_g}, \\
\mathfrak{p} \mathcal{O}_L&= \gamma_1^{e^{\prime}_1} \gamma_2^{e^{\prime}_2} \dots \gamma_t^{e^{\prime}_t},
\end{align*}
be the decomposition forms of $\mathfrak{p}$ in $K/P$ and $L/P$, respectively. For a positive integer $f$, 
let $\Pi_{\mathfrak{p}^f}(L/P)=\prod_{j=1}^d \gamma_j$, and $\{ \beta_1,\beta_2,\dots, \beta_s\}=\{\gamma_j \cap K : j=1,2,\dots,d\}$ be set of the all distinct  prime ideals of $K$ below $\gamma_j$'s. Also for every $i=1,\dots,s$, let $\{ \gamma_{i,1},\dots,\gamma_{i,u_i} \}$ be the set of the all distinct prime ideals of $L$ above $\beta_i$. Since $L/K$ is a Galois extension, for every $i=1,\dots,s$, the ideals $\gamma_{i,1},\gamma_{i,2},\dots,\gamma_{i,u_i}$ have the same ideal norm (over $K$), say $\beta_i^{\mathfrak{f}_i}$. Hence
\begin{equation*}
\Pi_{\mathfrak{p}^f}(L/P)=\prod_{i=1}^s \Pi_{\beta_i^{\mathfrak{f}_i}}(L/K) \in \Po(L/K).
\end{equation*}
\end{proof}

\begin{remark}
For a finite extension $L/K$ of Galois number fields, Chabert  proved that $\epsilon_{L/K}(\Po(K)) \subseteq \Po(L)$, see  \cite[Proposition 3.4]{Chabert 2001}; part $(i)$ in Lemma \eqref{lemma, for P sub N sub M, if M/N is Galois are Galois, then Po(M/P) is contained in Po(M/N)} is a relativization of this result. 
\end{remark}

\begin{remark}
Note that if either $K/P$ or $L/P$ is not Galois, then the containment in the part $(i)$ in Lemma \eqref{lemma, for P sub N sub M, if M/N is Galois are Galois, then Po(M/P) is contained in Po(M/N)}  might not hold. For instance, consider the pure cubic field $K=\mathbb{Q}(\sqrt[3]{19})$. One can show that the Galois closure $L$ of $K$ is a P\'olya field while $\Po(K)=\Cl(K)\simeq \mathbb{Z}/3 \mathbb{Z}$, see \cite[Example 2.9]{Maarefparvar-Rajaei}. On the other hand since $gcd(h(K),[L:K])=1$, $\mathcal{N}_{L/K} \, o \, {\epsilon}_{L/K}:\bar{\mathfrak{a}} \in \Cl(K) \mapsto \bar{\mathfrak{a}}^{[L:K]} \in \Cl(K)$ is injective, so is ${\epsilon}_{L/K}$. Therefore ${\epsilon}_{L/K}(\Po(K))=\Po(K)=\Cl(K) \not \subseteq \Po(L)$. 
\end{remark}

\begin{remark}
 By Corollary \eqref{corollary, if M/N is Galois and unramified and every ideal class of N extended to M is principal, then Cl(N) isomorphic to H^1 and Po(M/N) is trivia} and Lemma \eqref{lemma, for P sub N sub M, if M/N is Galois are Galois, then Po(M/P) is contained in Po(M/N)}, $\Po(H(H(K))/K)=\{0\}$. Hence every number field in the tower of Hilbert class fields for $K$ has trivial relative P\'olya group over $K$.
\end{remark}


Leriche also proved that the genus field  $\Gamma(K)$ of an abelian number field $K$ is  P\'olya , see \cite[Theorem 3.8]{Leriche 2014}.  But unlike the Hilbert class field, the relative P\'olya group of the genus field, even for abelian number fields, is not necessarily trivial: 

\begin{example}\label{example, relative Polya group of the genus field is not trivial}
For $K=\mathbb{Q}(\sqrt{-23})$, we have $h(K)=3$. By \cite[Lemma 3]{T. Honda}, the Hilbert class field $H(K)$ is a Galois extension of $\mathbb{Q}$. Since $H(K)/K$ is unramified, one can easily show that $H(K)/\mathbb{Q}$ cannot be cyclic (see proof of Lemma \eqref{lemma, necessary and sufficient conditions to class number of E divisible by r}),  so $Gal(H(K)/\mathbb{Q}) \simeq S_3$. On the other hand any compositum of $K$ and a cyclic cubic extension of $\mathbb{Q}$ is abelian. Thus $\Gamma(K)=K$, and $\Po(\Gamma(K)/K)=\Cl(K)\simeq \mathbb{Z}/3\mathbb{Z}$. 
\end{example}



\section{Main Result} \label{section, Main Result}

 From now on, $K$ is a non-Galois number field of degree $l$ (an odd prime) whose Galois closure has  Galois group isomorphic to $D_l$, the dihedral group of odrer $2l$. Denote the Galois closure of $K$ by $L$ and the unique quadratic subfield of $L$ by $E$.
  

 Suppose that $p$ is a ramified prime  in $L$, with decomposition in $L$ by $p \mathcal{O}_L=(\gamma_1 \gamma_2 \dots \gamma_g)^{e(p)}$. Since $e(p)f(p)g=[L:\mathbb{Q}]=2l$, we have $e(p)=2$ or $e(p)=l$ or $e(p)=2l$, where $f(p)$ is the residue class degree of $\gamma_i$'s over $p$. 
 
With these notations, following \cite{H. Cohen Adv.} we restate the complete description of decomposition forms of ramified primes $p$ in $K$ and $L$:

\begin{proposition} \label{proposition, decomposition form of p in K and L, H. Cohen Adv.}
With $K$ and $L$ as above:
\begin{itemize} 
\item[(1)] If $e(p)=2$, then $f(p)=1$. Moreover, if
\begin{equation} \label{eq. 2.1}
p \mathcal{O}_L=\gamma_1^2\gamma_2^2 \dots \gamma_l^2 
\end{equation}
 is the decomposition of $p$ in  $L$, then the decomposition of $p$ in $K$ has the form below:
\begin{equation} \label{eq. 2}
p \mathcal{O}_K=\beta_1 \beta_2^2 \dots \beta_{\frac{l+1}{2}}^2.
\end{equation}
\item[(2)] If $e(p)=l$, then $f(p)=1$ or $f(p)=2$ and $p$ is totally ramified in $K$.
\item[(3)] If $e(p)=2l$, then $p=l$ and it is totally ramified in $K$.
\end{itemize}
\end{proposition}

\begin{proof}
A detailed analysis of ramification groups using \cite[Chapter III]{Serre} yields the claims, for details see \cite[Proposition 10.1.26]{H. Cohen Adv.}.
\end{proof}
 



\begin{theorem} \label{theorem, main Theorem}
With $E$, $K$, and $L$ as above, we have an exact sequence as follows:
\begin{equation} \label{equation, exact sequence for Polya groups E, K and L}
\{0\} \longrightarrow \Po(E) \longrightarrow \Po(L) \longrightarrow \Po(K).
\end{equation}
Moreover, the $2$-torsion subgroup of $\Po(L)$ is isomorphic to $\Po(E)$ and the $l$-torsion subgroup of $\Po(L)$ is embedded in $\Po(K)$. In particular, if $l \not | \, h(K)$, then $\Po(E) \simeq \Po(L)$.
\end{theorem}

\begin{proof}
By Zantema's exact sequence \eqref{equation, Zantema's exact sequence}, $\Po(L)$ is a $2l$-torsion group. 
 Also $\Po(L)$ is the direct sum of its $2$-torsion group $\Po(L)_2$ and its $l$-torsion group $\Po(L)_l$. 
 By Lemma \eqref{lemma, for P sub N sub M, if M/N is Galois are Galois, then Po(M/P) is contained in Po(M/N)}, $\epsilon_{L/E}(\Po(E)) \subseteq \Po(L)$. Since $\Po(E)$ is a $2$-torsion group,
 \begin{equation*}
 \mathcal{N}_{L/E} o \epsilon_{L/E}|_{\Po(E)}: [\Pi_p(E)] \in \Po(E) \mapsto [\Pi_p(E)]^l \in \Po(E)
 \end{equation*}
is injective, and so is  $\epsilon_{L/E}|_{\Po(E)} : \Po(E) \rightarrow \Po(L)$.

 On the other hand, by Proposition \eqref{proposition, decomposition form of p in K and L, H. Cohen Adv.}, one can easily show that for an ideal class $[\Pi_p(L)] \in \Po(L)_2$,  we have $\mathcal{N}_{L/E}([\Pi_p(L)])=([\Pi_p(E)])^l \in \Po(E)$.
Hence $\mathcal{N}_{L/E}(\Po(L)_2) \subseteq \Po(E)$, and again since $\Po(E)$ is a $2$-torsion group, 
\begin{equation*}
\mathcal{N}_{L/E}|_{\Po(L)_2}:\Po(L)_2 \rightarrow \Po(E)
\end{equation*}
 is injective.  Therefore $\Po(E) \simeq \Po(L)_2$ and we find the following exact sequence:
\begin{displaymath}
\xymatrix{
\{0\} \ar[r]  & \Po(E)  \ar[r]^{\epsilon_{L/E}}  & \Po(L) \ar[r]^{\pi_l} & \Po(L)_l \ar[r] & \{0\}, 
 }
\end{displaymath}

where $\pi_l: \Po(L) \rightarrow \Po(L)_l$ is the projection map (i.e. multiplication by $2$). 

\vspace*{0.3cm}

We claim that  $\mathcal{N}_{L/K}(\pi_l(\Po(L)) \subseteq \Po(K)$ and $\mathcal{N}_{L/K}|_{\Po(L)_l}$ is also  injective. To prove these, let $p$ be a ramified prime  in $L$ such that $l|e(p)$. By Proposition \eqref{proposition, decomposition form of p in K and L, H. Cohen Adv.}, $p$ is totally ramified in $K$, say $p \mathcal{O}_K=\beta^l=(\Pi_p(K))^l$. Two cases are possible:

\begin{itemize}
\item
If $e(p)=l$ then by Proposition \eqref{proposition, decomposition form of p in K and L, H. Cohen Adv.}, one has $p \mathcal{O}_L=(\gamma_1 \gamma_2)^l=(\Pi_p(L))^l$ or $p \mathcal{O}_L=\gamma^l=(\Pi_{p^2}(L))^l$. In this case, depending on whether $f(p)=1$ or $f(p)=2$,  $\mathcal{N}_{L/K}([\Pi_p(L)])=[\Pi_p(K)]^2 \in \Po(K)$ or $\mathcal{N}_{L/K}([\Pi_{p^2}(L)])=[\Pi_p(K)]^2 \in \Po(K)$, respectively. Also if $(\Pi_p(K))^2$ is principal, then $\Pi_p(K)$ must be principal, since $(\Pi_p(K))^l$ is also principal. On the other hand $j_{L/K}(\Pi_p(K))=\Pi_p(L)$ (resp. $\Pi_{p^2}(L)$), which implies that  $\Pi_p(L)$ (resp. $\Pi_{p^2}(L)$) is also principal.

\vspace*{0.3cm}
\item
If $e(p)=2l$, then by Proposition \eqref{proposition, decomposition form of p in K and L, H. Cohen Adv.}, $p=l$ ramifies totally in $L$ and so in all its subextensions. In this case $N_{L/K}(\Pi_l(L))=\Pi_l(K)$ and if  $\Pi_l(K)$ is principal, then $j_{L/K}(\Pi_l(K))=(\Pi_l(L))^2$ is principal which implies that the ideal class $[\Pi_l(L)]$ belongs to $\Po(L)_2$. In other words, $\pi_l[\Pi_l(L)] \in \Po(L)_l \simeq \frac{\Po(L)}{\epsilon_{L/E}(\Po(E))}$ would be trivial.

\end{itemize}



Summing up the above arguments, we have proved that 
\begin{equation*}
\mathcal{N}_{L/K}|_{\Po(L)_l}:\Po(L)_l \rightarrow \Po(K)
\end{equation*}
is injective, as claimed. Therefore the following sequence is exact: 
\begin{displaymath}
\xymatrix{
\{0\} \ar[r]  & \Po(E)  \ar[r]^{\epsilon_{L/E}}  & \Po(L) \ar[r]^{\varphi} & \Po(K) ,
 }
\end{displaymath}
where $\varphi:=\mathcal{N}_{L/K}|_{\Po(L)_l}\, o \, \pi_l$.
\end{proof}

\begin{remark}  \label{remark, Zantema's result about composite of Galois Polya fields}
Zantema  \cite{zantema}, proved that for two finite Galois extensions $K_1$ and $K_2$ of $\mathbb{Q}$ with $M=K_1.K_2$, if for every prime number $p$, the ramification indices of $p$ in $K_1$ and $K_2$ are coprime, then P\'olya-ness of $K_1$ and $K_2$ implies that $M$ is also P\'olya. Conversely, if $\gcd([K_1:\mathbb{Q}],[K_2:\mathbb{Q}])=1$ and $M$ is P\'olya then $K_1$ and $K_2$ are P\'olya, see \cite[Theorem 3.4]{zantema}. Under these hypotheses, 
 one can easily show that $\epsilon_{M/K_1}(\Po(K_1)).\epsilon_{M/K_2}(\Po(K_2))=\Po(M)$, see 
\cite{Chabert 2001}. The condition on relative primality of the degrees is necessary as was shown in \cite{HR1} and \cite{HR2} in the case of biquadratic fields.
Also the condition on Galois-ness of both $K_1$ and $K_2$ is necessary: 

With the notations in this section, for every ramified prime $p$ in $L$ with $e(p)=2$, by Proposition \eqref{proposition, decomposition form of p in K and L, H. Cohen Adv.}, $\Pi_p(K)=\beta_1 \beta_2 \dots \beta_{\frac{l+1}{2}}$, and $j_{L/K}(\Pi_p(K))=\gamma_1 \Pi_p(L)$. Since $\Gal(L/\mathbb{Q})$ acts transitively on the set $\{\gamma_1,\dots,\gamma_l\}$, $\gamma_1$ is not ambiguous. Hence $\epsilon_{L/K}(\Po(K)) \not \subseteq \Po(L)$.
\end{remark}







\begin{corollary} \label{corollary, if L/E is unramified then Po(E) isomorph. to Po(L)}
If all finite places of $E$ are unramified in $L$, then $\Po(L) \simeq \Po(E)$.
\end{corollary}

\begin{proof}
We claim if $L/E$ is unramified, then every ramified prime $p$ in $L/\mathbb{Q}$ has ramification index $2$. Otherwise let $p$ be a ramified prime in $L/\mathbb{Q}$ with $l | e(p)$, where $e(p)$ denotes the ramification index of $p$ in $L/\mathbb{Q}$. Let $\mathfrak{p}$ and $\gamma$ be prime ideals in $E$ and $L$ above $p$, respectively. Since
$l | e(p)=e(\gamma/\mathfrak{p}). e(\mathfrak{p}/p)$,  $\mathfrak{p}$ ramifies in $L/E$ and we reach a contradiction. Hence  the $l$-torsion subgroup $\Po(L)_l$ of $\Po(L)$ would be trivial and the statement follows from Theorem \eqref{theorem, main Theorem}. 
\end{proof}

If $L/E$ is unramified, by class field theory $h(E)$ is divisible by $l$. 
 By an argument  similar to Honda's result for $l=3$ \cite[Proposition 10]{T. Honda}, we give necessary and sufficient conditions for divisibility of class number of a quadratic field by $l$:

\begin{lemma} \label{lemma, necessary and sufficient conditions to class number of E divisible by r}
Suppose that $N$ is a quadratic field and $l|h(N)$. Then there exists a $D_l$-extension $M$ of $\mathbb{Q}$ such that $N$ is the quadratic subfield of $M$, and $M$ is unramified over $N$. Conversely, let $M$ be a $D_l$-extension of $\mathbb{Q}$ which is the splitting field of the irreducible polynomial 
\begin{equation}
f(X)=X^l+a_2 X^{l-2}+a_3 X^{l-3}+ \dots +a_{l-1} X +a_l , \quad a_i \in \mathbb{Z} \label{eq. 2.6}
\end{equation}
over $\mathbb{Q}$. If $\gcd(a_2,a_3, \dots, a_{l-1}, l.a_l)=1$, then class number of the unique quadratic subfield of $M$ is divisible by $l$.
\end{lemma}

\begin{proof}
Assume that $h(N)$ is divisible by $l$. Hence there exists an unramified abelian extension $M$ of degree $l$ over $N$. One can show that $M$ is a Galois extension of $\mathbb{Q}$, thanks to  Honda's lemma \cite[Lemma 3]{T. Honda}.
If $\Gal(M/\mathbb{Q})\simeq C_{2l}$, the cylic group of order $2l$, denote the unique subfield of $M$ of degree $l$ over $\mathbb{Q}$ by $F$. Hence $F$ is a cyclic extension of $\mathbb{Q}$, and so every ramified prime $p$  in the extension $F/\mathbb{Q}$, is totally ramified. Therefore, $p$ has the ramification index $l$ or $2l$ in the extension $M/\mathbb{Q}$. In both cases, there exists a prime ideal of $N$, above $p$ which ramifiies in the extension $M/N$ and we reach a contradiction.
Thus $M$ is not abelian over $\mathbb{Q}$, i.e. $\Gal(M/\mathbb{Q}) \simeq D_l$. 

Conversely, let $M$ be a $D_l$-extension of $\mathbb{Q}$. Let $N$ be the unique quadratic subfield of $M$ and denote
a subfield of $M$ of degree $l$ over $\mathbb{Q}$ by $F$. Let $\mathfrak{p}$ be a prime ideal of $N$ ramified in $M$. Hence for $p=\mathfrak{p}\cap \mathbb{Q}$, its ramification index is divisible by $l$. By Proposition \eqref{proposition, decomposition form of p in K and L, H. Cohen Adv.}, $p$  totally ramifies in $F/\mathbb{Q}$ which implies that there exists  $f(X)\in \mathbb{Z}[X]$ with 
\begin{equation*}
f(X) \equiv (X-h)^l ({\mathrm{mod}\, p}),
\end{equation*}
where $h \in \mathbb{Z}$ and $M$ is the splitting field of $f(X)$ over $\mathbb{Q}$. From this congruence equation, we have either 
\begin{equation*}
p=l \quad \text{and} \quad l \mid \gcd(a_2,a_3, \dots, a_{l-1}),
\end{equation*}
or
\begin{equation*}
p \mid \gcd(a_2,a_3, \dots, a_{l-1},a_l).
\end{equation*}
Thus $p$ divides $\gcd(a_2,a_3, \dots, a_{l-1},l.a_l)$. On the other hand, if there exists no totally ramified prime in the extension $F/\mathbb{Q}$, by Proposition \eqref{proposition, decomposition form of p in K and L, H. Cohen Adv.} every ramified prime in $M/\mathbb{Q}$ has ramification index $2$, which implies that $M/N$ is unramified. 
\end{proof}



Using Corollary \eqref{corollary, if L/E is unramified then Po(E) isomorph. to Po(L)} and Lemma \eqref{lemma, necessary and sufficient conditions to class number of E divisible by r}, we obtain:

\begin{corollary} \label{corollary, if gcd(a,3b)=1 and E is Polya, then so is L}
Assume that $L$ is a $D_l$-extension of $\mathbb{Q}$ and denote the unique quadratic subfield of $L$ by $E$. Let $L$ be the splitting field of an irreducible polynomial 
\begin{equation*}
f(X)=X^l+a_2 X^{l-2}+a_3 X^{l-3}+ \dots +a_{l-1} X +a_l , \quad a_i \in \mathbb{Z} 
\end{equation*}
over $\mathbb{Q}$. If $\gcd(a_2,a_3, \dots, a_{l-1}, l.a_l)=1$, then $\Po(E) \simeq \Po(L)$.
\end{corollary}

Considering the decomposition form of a prime $p$ in $L$ and $K$, given in Proposition \eqref{proposition, decomposition form of p in K and L, H. Cohen Adv.}, we find also a result on divisibilty of class numbers.

\begin{proposition} \cite[Proposition 4.4]{NChildress} \label{proposition, Childress's prop.}
If $M/N$ is an extension of number fields, and there is some prime $\mathfrak{p}$ of $\mathcal{O}_N$ that is totally ramified in $M/N$, then $h(N) | h(M)$. 
\end{proposition}

\begin{corollary}  \label{corollary, h(K) divides h(L)}
Let $K$ be a non-Galois number field of degree $l$, for $l$ an odd prime. Let $L$ be the Galois closure of $K$ over $\mathbb{Q}$ with $\Gal(L/\mathbb{Q}) \simeq D_l$. Then $h(K) | h(L)$. In particular, if  $h(L)=1$, then both subfields $K$ and the unique quadratic subfield $E$ of $L$, are P\'olya.
\end{corollary}

\begin{proof}
Let $p$ be a ramified prime in the extension $E/\mathbb{Q}$. Hence $2 \mid e(p)$, where $e(p)$ denotes the ramification index of $p$ in $L/\mathbb{Q}$.
If $e(p)=2$, by Proposition \eqref{proposition, decomposition form of p in K and L, H. Cohen Adv.}, $p$ has the decomposition form in $K$ as follows:
\begin{equation} 
p \mathcal{O}_K=\beta_1 \beta_2^2 \dots \beta_{\frac{l+1}{2}}^2,
\end{equation}
which implies that $\beta_1$ ramifies in $L/K$. Similarly, if $e(p)=2l$, then by Proposition \eqref{proposition, decomposition form of p in K and L, H. Cohen Adv.}, $p=l$ and in this case the only prime ideal $\beta$ of $K$ above $l$ is ramified in $L/K$.
Hence by Proposition \eqref{proposition, Childress's prop.}, $h(K) | h(L)$. The second assertion follows from Theorem \eqref{theorem, main Theorem}.
\end{proof}

\vspace*{0.3cm}

At the end of this section, it might be appropriate to have some numerical examples. To find $D_5$-extensions,  we use \textit{Brumer's generic polynomial} (see \cite[Definition 0.1.1]{C. U. Jensen's book}):

\vspace*{0.1cm}

\begin{proposition} \cite[Theorem 2.3.5]{C. U. Jensen's book} \label{Brumer's theorem}
Let $M$ be an arbitrary field. The polynomial 
\begin{equation} \label{generic polynomial for D_5-extension}
f(s,t,X)=X^5+(t-3)X^4+(s-t+3)X^3+(t^2-t-2s-1)X^2+sX+t 
\end{equation}
in $M(s,t)[X]$ is then generic for $D_5$-extensions over $M$.

Also the quadratic subextension (in chracteristic $\neq 2$) of the splitting field of the polynomial $f(s,t, X)$ is obtained by adjoining to $M(s,t)$ a square root of 
\begin{equation*}
-(4t^5-4t^4-24st^3-40t^3-s^2t^2+34st^2+91t^2+30s^2t+14st-4t-s^2+4s^3).
\end{equation*}
\end{proposition}


\vspace*{0.3cm}

\begin{example} \label{example, examples of D_5-extension of Q}
Consider the notations of Proposition \eqref{Brumer's theorem}.

\begin{itemize}

\item[(a)]
Let $s=5$ and $t=1$ and $K=\mathbb{Q}(\theta)$ where $\theta$ is a root of the polynomial
\begin{equation*}
f(X)=X^5-2X^4+7X^3-11X^2+5X+1.
\end{equation*}
We have $D_K=1367^2$ and $K$ is a $D_5$-field. Denote the Galois closure of $K$ over $\mathbb{Q}$ by $L$. By Proposition \eqref{Brumer's theorem}, $E=\mathbb{Q}(\sqrt{-1367})$ is the unique quadratic subfield of $L$ and by Proposition \eqref{proposition, Zantema's result for quadratic Polya field}, $E$ is P\'olya. Since $h(K)=4$, By 
Theorem \eqref{theorem, main Theorem}, $L$ is a P\'olya $D_5$-extension of $\mathbb{Q}$, while by \cite[Theorem 6.9]{zantema} $K$ is not P\'olya.

\item[(b)]
for $s=-5$ and $t=3$ the splitting field $L$ of the polynomial 
\begin{equation*}
f(X)=X^5 - 5X^3 + 15X^2 - 5X+ 3.
\end{equation*}
over $\mathbb{Q}$ is a $D_5$-extension with the uniuqe quadratic subfield $E=\mathbb{Q}(\sqrt{-15})$. We have $D_K=3^2.5^6$, $h(K)=1$, where $K=\mathbb{Q}(\theta)$ for some root $\theta$ of $f(X)$. Using Proposition \eqref{proposition, Zantema's result for quadratic Polya field} and Theorem \eqref{theorem, main Theorem}, we have $\Po(L) \simeq \Po(E) \simeq C_2$.


\end{itemize}
\end{example}
%



We recall that for a $D_l$-extension $L$ of $\mathbb{Q}$ with quadratic subfield $E$, if $L/E$ is unramified then $\Po(E) \simeq \Po(L)$, see Corollary \eqref{corollary, if L/E is unramified then Po(E) isomorph. to Po(L)}. 

\begin{proposition} \cite[Section 2]{M. J. Lavallee} \label{Lavallee's criterion}
For $s \in \mathbb{Z}$ let 
\begin{equation} \label{Lavallee's polynomial f_s(x)}
f_s(X)=X^5-2X^4+(s+2)X^3-(2s+1)X^2+sX+1.
\end{equation}
Then, $f_s(X)$ is irreducible over $\mathbb{Q}$, $\disc(f_s(X)) = (4s^3+28s^2+24s+47)^2$ and if $-(4s^3+28s^2+24s+47)$ is not a square, then $\Gal(f(X))\simeq D_5$. Moreover, for $\Gal(f(X))\simeq D_5$
the splitting field of $f(X)$ (over $\mathbb{Q}$) is unramified over its unique quadratic subfield, namely over $\mathbb{Q}(\sqrt{-4s^3-28s^2-24s-47})$.
\end{proposition}




\begin{example} \label{example, some examples of Polya and non-Polya D_5-extensions of Lavallee's polynomials f_s(x)}
Let $f_s(X)$ be given by equation \eqref{Lavallee's polynomial f_s(x)}. Using Corollary \eqref{corollary, if L/E is unramified then Po(E) isomorph. to Po(L)} one can easily check that the splitting field of $f_s(X)$ for  $s\in\{ \pm6,\pm5,\pm2,\pm1,0,8 \}$
is a P\'olya $D_5$-extension of $\mathbb{Q}$, while for $s \in\{ \pm17,\pm16,\pm4,-3,7,9,10 \}$ is not.
\end{example}


Following \cite{C. U. Jensen's book} we give also an example of a P\'olya $D_7$-extension of $\mathbb{Q}$:

 \begin{example} \label{example, first example of $D_7$-extension of Q}
Let $K=\mathbb{Q}(\alpha)$ where $\alpha$ is a root of
\begin{equation*}
f(X)=X^7-7X^6-7X^5-7X^4-1.
\end{equation*}
 Then $f(X)$ is irreducible over $\mathbb{Q}$ and the splitting field $L$ of $f(X)$ (over $\mathbb{Q}$) has Galois group isomorphic to $D_7$, see \cite[Section 5.2, Example 5]{C. U. Jensen's book}. 
Since $\disc(f(X))=-3^6.7^9$, $E=\mathbb{Q}(\sqrt{-7})$ is the quadratic subfield of $L$, which is  P\'olya by Proposition \eqref{proposition, Zantema's result for quadratic Polya field}. Also  $h(K)=1$, and by Theorem  \eqref{theorem, main Theorem} $L$ is P\'olya. 
\end{example}

  
\begin{remark}
In \cite[Section 7.3]{C. U. Jensen's book}, using singular values of certain modular functions, a method for finding dihedral extensions of $\mathbb{Q}$ and some algorithms for constructing the Hilbert class field of an imaginary quadratic field are given. In particular, using Corollary \eqref{corollary, if L/E is unramified then Po(E) isomorph. to Po(L)}, one can see that examples (2) and (3) of \cite[Section 7.3]{C. U. Jensen's book} are respectively $D_{13}$ and $D_{19}$-extensions of $\mathbb{Q}$. This can also be done by Leriche's result \cite[Corollary 3.2]{Leriche 2014}.
\end{remark}
 



\section{Upper bound for the number of ramification} \label{section of Maximum number of ramification}
We again recall that $l$ is an odd prime number.
For a number field $M$, denote the number of ramified primes in $M/\mathbb{Q}$ by $s_M$. 
Leriche \cite{Leriche 2013}, for any  Galois P\'olya number field $M$ gave an upper bound for $s_M$ which only depends on the degree of $M$ over $\mathbb{Q}$, see  \cite[Proposition 2.5]{Leriche 2013}. In particular,
for a P\'olya Galois number field $M$ of degree $2l$, this upper bound is given by $s_M\leq l+4$. 

For a cyclic number field $M$ of degree $2l$, using Zantema's results \cite[Proposition 3.2 and Theorem 3.4]{zantema} the upper bound can be made sharp by $s_M \leq 3$.


In \cite{Maarefparvar-Rajaei}, we proved that for a non-Galois cubic field $K$ with Galois closure $L$, if $L$ is P\'olya depending on whether $D_K>0$, $D_K<0$ and $K$ pure, or $D_K<0$ and $K$ non-pure, then $s_L \leq 4$, $s_L \leq 3$ or $s_L \leq 2$, respectively.  Also by giving some examples, we showed that these upper bounds are actually sharp, see \cite[Section 3]{Maarefparvar-Rajaei}.

 In this section, by using the same methods in \cite{Maarefparvar-Rajaei}, for a P\'olya $D_l$-extension $L$ of $\mathbb{Q}$ we give an upper bound for $s_L$ which is much smaller than $l+4$.

\begin{theorem} \label{theorem, maximum number of ramification in Polya dihedral extension}
Let $K$ be a non-Galois number field of prime degree $l > 3$. Let $L$ be the Galois closure of $K$ over $\mathbb{Q}$ with $\Gal(L/\mathbb{Q}) \simeq D_l$. Denote the unique quadratic subfield of $L$ by $E$, and denote the norm morphism of $L^{\times}$ to $K^{\times}$ and $E^{\times}$
 by  $Nm_{L/K}$ and $Nm_{L/E}$, respectively. If $L$ is P\'olya, then:
\begin{itemize}
\item [(i)] For $L$ real, depending on whether the fundamental unit of $E$ belongs to $Nm_{L/E}(U_L)$ or not, $s_L \leq 3$ or  $s_L \leq 4$, respectively. 
\item [(ii)] For $L$ imaginary, $s_L \leq 2$.
\end{itemize}
\end{theorem}

\begin{proof}
Let $G=\Gal(L/\mathbb{Q}) \simeq D_l$. As before, for a ramified prime $p$ in $L$, denote its ramification index  by $e(p)$. Since $L$ is a P\'olya field, by exact sequence \eqref{equation, Zantema's exact sequence}, we have
\begin{equation} \label{equation, for Polya Galois extension L of Q, order of H^1(G,U_L)=product of ramification indexes}
\#H^1(G, U_L)=\prod_{p \mid D_L} e(p),
\end{equation}
which implies that $\#H^1(G, U_L)$ is  a divisor of a power of $2l$. Now consider the cyclic extensions $L/K$ and $L/E$. Let $G_2=\Gal(L/K)$ and $G_l=\Gal(L/E)$ and use the \textit{Herbrand quotients}:
\begin{align} \label{equation, Herbrand quotients of L over K and E}
Q(G_2,U_L)=\frac{\# \hat{H}^0(G_2, U_L)}{\#{H}^1(G_2,U_L)}, \quad
Q(G_l,U_L)=\frac{\# \hat{H}^0(G_l, U_L)}{\#{H}^1(G_l,U_L)},
\end{align}
where 
\begin{align*}
\hat{H}^0(G_2, U_L)&=U_L^{G_2}/Nm_{L/K}(U_L)=U_K/Nm_{L/K}(U_L), \\
\hat{H}^0(G_l, U_L)&=U_L^{G_l}/Nm_{L/E}(U_L)=U_E/Nm_{L/E}(U_L).
\end{align*}

On the other hand, the Herbrand quotients $Q(G_2,U_L)$ and $Q(G_l,U_L)$ are given by \cite[Proposition 5.10]{NChildress}:
\begin{align*}
Q(G_2,U_L)&= \frac{2^s}{[L:K]}=2^{s-1}, \\
Q(G_l,U_L)&= \frac{2^t}{[L:E]}=\frac{2^t}{l},
\end{align*}
where $s$ (resp. $t$) is the number of infinite places of $K$ (resp. $E$) ramified in $L$. For $L$ real (resp. imaginary), the signature of $K$ is $(l,0)$ (resp. $(1,\frac{l-1}{2})$), see  \cite[Theorem 9.2.6]{H. Cohen Adv.}. Hence
\begin{align}
Q(G_2,U_L)&=\left \{
\begin{array}{ll}
\frac{1}{2} & :  \text{L is real,} \\
1 & : \text{L is imaginary,} \\
\end{array} \label{equation, Herbrand quotient of L over K}
\right.
 \\ Q(G_l,U_L)&=\frac{1}{l}. \label{equation, Herbrand quotient of L over E}
\end{align}

For the cyclic extension $L/K$, since $Nm_{L/K}(U_L)$ contains $U_K^2$ and  $(U_K:U_K^2)$ divides $2^{w_K+1}$, $(U_K:Nm_{L/K}(U_L))$ divides $2^{w_K+1}$, where $w_K$ denotes the Dirichlet rank of group of units of $K$.

Similarly, for the cylic extension $L/E$, $(U_E:Nm_{L/E}(U_L))$ divides $l^{w_E+1}$. But for $E=\mathbb{Q}(\sqrt{d})$ imaginary, for $d \not \in \{-1,-3 \}$, $d=-1$ or $d=-3$ we have $U_E=\{\pm1\}$, $U_E=\{\pm1,\pm i \}$ or $U_E=\{ \pm 1, \pm \zeta_3 , \pm \zeta_3^2 \}$, respectively, where $\zeta_3=e^{\frac{2\pi i}{3}}$.
Thus for $E$ imaginary, since $(U_E:Nm_{L/E}(U_L))$ divides $\#U_E$ and $\gcd(\#U_E,l)=1$, $(U_E:Nm_{L/E}(U_L))=1$.

Now let $E$ be real, and $\xi$ be the fundamental unit of $E$.  By Dirichlet Unit Theorem we have $U_E\simeq C_2 \oplus \mathbb{Z}$. Hence depending on whether $\xi \in Nm_{L/E}(U_L)$ or not, $(U_E:Nm_{L/E}(U_L))=1$ or  $(U_E:Nm_{L/E}(U_L))=l$, respectively. 

Summing up the above arguments, and using relations \eqref{equation, Herbrand quotients of L over K and E}, \eqref{equation, Herbrand quotient of L over K} and \eqref{equation, Herbrand quotient of L over E}, we find:
\begin{itemize}
\item for $L$ real, $\# H^1(G_2,U_L) \mid 2^{l+1}$ and depending on whether the fundamental unit of $E$ belongs to $Nm_{L/E}(U_L)$ or not, $\# H^1(G_l,U_L)=l$ or $\# H^1(G_l,U_L)=l^2$, respectively.
\item for $L$ imaginary, $\# H^1(G_2,U_L) \mid 2^{\frac{l+1}{2}}$ and $\# H^1(G_l,U_L)=l$.
\end{itemize}

On the other hand, the restriction maps
\begin{equation*}
\res:H^1(G, U_L) \rightarrow H^1(G_2,U_L),  
\end{equation*}
and
\begin{equation*}
\res:H^1(G, U_L) \rightarrow H^1(G_l,U_L),
\end{equation*}
are injective on the $2$-primary and $l$-primary part of $H^1(G,U_L)$, respectively, see \cite[Proposition 1.6.9]{Neukirch's book}. Thus:
\begin{equation} \label{order of H^1 depends on order of it's sylow subgroups}
\# H^1(G,U_L) \mid \# H^1(G_2,U_L) . \# H^1(G_l,U_L).
\end{equation}
Therefore,
\begin{itemize}
\item for $L$ real, depending on whether the fundamental unit of $E$ belongs to $Nm_{L/E}(U_L)$ or not,
\begin{equation} \label{equation, for L real, order of H^1(G,U_L) divides 2^(r+1).r^2}
\#H^1(G,U_L) \mid 2^{l+1}. l^1 \quad \text{or} \quad \#H^1(G,U_L) \mid 2^{l+1}. l^2,
\end{equation}
respectively.
\item for $L$ imaginary, 
\begin{equation} \label{equation, for L imaginary, order of H^1(G,U_L) divides 2^(r+1/2).r^1}
\#H^1(G,U_L) \mid 2^{\frac{l+1}{2}}. l^1 .
\end{equation}
\end{itemize}

Now since $L$ is P\'olya by Theorem \eqref{theorem, main Theorem}, $E$ is also P\'olya. On the other hand, by Proposition \eqref{proposition, Zantema's result for quadratic Polya field}, for real (resp. imaginary) P\'olya field $E$, $s_E \leq 2$ (resp. $s_E=1$). Hence by Theorem \eqref{theorem, main Theorem}, the $2$-torsion subgroup of $\Po(L)$ has at most $2$ (resp. $1$) cyclic factors. Using the relations \eqref{equation, for L real, order of H^1(G,U_L) divides 2^(r+1).r^2} and \eqref{equation, for L imaginary, order of H^1(G,U_L) divides 2^(r+1/2).r^1} we find:
 
\begin{itemize}
\item for real P\'olya $D_l$-extension $L$ of $\mathbb{Q}$, depending on whether the fundamental unit of $E$ belongs to $Nm_{L/E}(U_L)$ or not,
\begin{equation} \label{equation, for L real and Polya, order of H^1(G,U_L) divides 2^2. r^2}
\#H^1(G,U_L) \mid 2^2.l^1 \quad \text{or} \quad \#H^1(G,U_L) \mid 2^2. l^2,
\end{equation}
respectively.
\item for imaginary P\'olya $D_l$-extension $L$ of $\mathbb{Q}$, 
\begin{equation} \label{equation, for L imaginary and Polya, order of H^1(G,U_L) divides 2^1. r^1}
\#H^1(G,U_L) \mid 2^1. l^1.
\end{equation}
\end{itemize}

Finally using relations \eqref{equation, for Polya Galois extension L of Q, order of H^1(G,U_L)=product of ramification indexes}, \eqref{equation, for L real and Polya, order of H^1(G,U_L) divides 2^2. r^2} and \eqref{equation, for L imaginary and Polya, order of H^1(G,U_L) divides 2^1. r^1}, the statement in theorem is proved.
\end{proof}

In \cite{Ishida}, Ishida proved that for a non-pure number field $M$ of degree $l$, if number of the totally ramified primes in $M/\mathbb{Q}$ is more than the rank of the unit group $w_M$, then $l \mid h(M)$, see  \cite[Theorem 2]{Ishida}. 
The method used to prove
Theorem \eqref{theorem, maximum number of ramification in Polya dihedral extension}, can yield Ishida-type results:
\begin{corollary} \label{corollary, a sufficient condition to h(K) is divisible by 3 for S_3-field K}
Let $K$ be a non-Galois cubic field, and denote the number of totally ramified primes in $K/\mathbb{Q}$ by $t_K$.
Then:
\begin{itemize}
\item[$(i)$] for $D_K > 0$, if $t_K \geq 3$, then $3 \mid \#\Po(K)$;
\item[$(ii)$] for $D_K < 0$ and $K$ pure, if $t_K \geq 3$, then $3 \mid \#\Po(K)$;
\item[$(iii)$] for $D_K < 0$ and $K$ non-pure, if $t_K \geq 2$, then $3 \mid \#\Po(K)$.
\end{itemize}
In particular, in the above cases $3 | h(K)$.
\end{corollary}

\begin{proof}
As before, denote the Galois closure of $K$ over $\mathbb{Q}$ by $L$ and assume that $G=Gal(L/\mathbb{Q})$. In  \cite{Maarefparvar-Rajaei} we proved that 
\begin{itemize}
\item
for $L$ real,  $\# H^1(G,U_L) | 2^4.3^2$;
\item
for $L$ imaginary and $K$ pure, $\# H^1(G,U_L) | 2^2.3^2$;
\item
for $L$ imaginary and $K$ non-pure, $\# H^1(G,U_L) | 2^2.3^1$;
\end{itemize}

see \cite[proof of Theorem 3.1]{Maarefparvar-Rajaei}. Thereby the $3$-primary part of $H^1(G,U_L)$ for $L$ real, or $L$ imaginary with $K$ pure has order at most $9$, while for $L$ imaginary with $K$ non-pure has order at most $3$.
On the other hand by Zantema's exact sequence \eqref{equation, Zantema's exact sequence}, we have:
\begin{equation} \label{equation, relation between order of H^1, Po(L) and ramification index}
\# H^1(G,U_L). \#\Po(L) =\prod_{p | D_L} e(p).
\end{equation}

$(i),(ii)$. For $L$ real, or $L$ imaginary with $K$ pure, if  at least three distinct primes totally ramify in $K/\mathbb{Q}$, then by Proposition \eqref{proposition, decomposition form of p in K and L, H. Cohen Adv.}, $3^3$ divides $\prod_{p | D_L} e(p)$. By relation \eqref{equation, relation between order of H^1, Po(L) and ramification index} and the above arguments, $3 |\# \Po(L)$ which implies that the $3$-torsion subgroup $\Po(L)_3$ of $\Po(L)$ is nontrivial. By Theorem \eqref{theorem, main Theorem}, $\Po(L)_3$ is embedded in $\Po(K)$.

Part $(iii)$ can be proved similarly. 
\end{proof}

For $D_l$-fields $K$ with $l>3$, one can find a lower bound, independent of $[K:\mathbb{Q}]$, 
for $t_K$ making $h(K)$  divisible by $l$.
\begin{corollary}  \label{corollary, a sufficient condition to h(K) is divisible by r for D_r-field K with r>3}
Let $K$ be a non-Galois number field of prime degree $l> 3$, and assume that the Galois closure of $K$ over $\mathbb{Q}$ has Galois group isomorphic to $D_l$.
Denote the number of totally ramified primes in $K/\mathbb{Q}$ by $t_K$. Then:
\begin{itemize}
\item[(i)] for $D_K >0$, if $t_K\geq 3$, then $l \mid \#\Po(K)$. 
\item[(ii)] for $D_K<0$, if $t_K\geq 2$, then $l \mid \#\Po(K)$.
\end{itemize}
In particular, in the above cases $l | h(K)$.
\end{corollary}

\begin{proof}
By an argument similar to proof of Corollary \eqref{corollary, a sufficient condition to h(K) is divisible by 3 for S_3-field K}, and using relations \eqref{equation, for L real, order of H^1(G,U_L) divides 2^(r+1).r^2} and \eqref{equation, for L imaginary, order of H^1(G,U_L) divides 2^(r+1/2).r^1}, the statements are proved.
\end{proof}

\begin{remark}
For a $D_l$-field $K$ with Galois closure $L$ and  unique quadratic subfield $E$ of $L$,
one has $D_K=(D_E)^{\frac{l-1}{2}} f ^{l-1}$,  where $f$ is the conductor of $L$ over $E$. Moreover, a prime $p$ is totally ramified in $K/\mathbb{Q}$ if and only if $p \mid f$ and also for $p \mid \gcd(D_E,f)$, we have $p=l$, see \cite[Proposition 10.1.28]{H. Cohen Adv.}. Hence, for instance, for a pure cubic field $K=\mathbb{Q}(\sqrt[3]{m})$, where $m=ab^2$ with squarefree coprime integers $a >b \geq 1$, if the conductor $f$ of $L=\mathbb{Q}(\sqrt[3]{m},\zeta_3)$ over $E=\mathbb{Q}(\sqrt{-3})$, has more than two distinct prime divisors, then $3 \mid h(K)$. (Note that $f=ab$ if $m \equiv \pm 1\, (\mathrm{mod}\, 9)$, and $f=3ab$ otherwise.)
\end{remark}


\bibliographystyle{amsplain}

\end{document}